\documentclass{amsart}
\usepackage{amsfonts,amssymb,amsmath,euscript}
\usepackage{graphicx}
\usepackage[matrix,arrow,curve]{xy}
\usepackage{wrapfig}

  \usepackage{amssymb}
\usepackage{mathrsfs}
\usepackage{euscript}
\usepackage{lipsum}

    \newtheorem{thm}{Theorem}                     [section]
    \newtheorem{thm*}{Theorem}

    \newtheorem{lemma*}{Lemma}    

    \newtheorem{conj}{Conjecture}                 
    \newtheorem{defn}[thm]{Definition}                 

\newcommand*{\norm}[1]{\left\Vert#1\right\Vert}    

\newcommand{\hilb}{\mathcal H}

\newcommand{\Scal}[1]{\left\langle #1\right\rangle}               

\newcommand{\mb}{\mathbb}

\newcommand{\euR}{\EuScript R}

\newcommand{\clA}{\mathcal A}

\newcommand{\clK}{\mathcal K}
\newcommand{\clV}{\mathcal V}

\newcommand{\mbC}{\mathbb C}
\newcommand{\mbD}{\mathbb D}
\newcommand{\mbR}{\mathbb R}

\newcommand{\ran}{\operatorname{ran}}
\newcommand{\spec}{\operatorname{spec}}
\newcommand{\Id}{\operatorname{Id}}

\newcommand{\eps}{\varepsilon}

\begin{document}
\title[On essentially singular points]{Spectral flow inside essential spectrum VI: \\  on essentially singular points}
\author{Nurulla Azamov}
\address{Independent scholar, Adelaide, SA, Australia}
\email{azamovnurulla@gmail.com}
 \keywords{limiting absorption principle, singular spectrum, spectral flow}
 \subjclass[2000]{ 
     Primary 47A40;
 }
\begin{abstract} 
Let $H_0$  be a self-adjoint operator on a Hilbert space $\mathcal H$ endowed with a rigging $F,$ which is a zero-kernel closed operator from $\mathcal H$ to another Hilbert space $\mathcal K$
such that the sandwiched resolvent $F (H_0 - z)^{-1}F^*$ is compact. Assume that $H_0$ obeys the limiting absorption principle (LAP) in the sense that the norm limit 
$F (H_0 - \lambda - i0)^{-1}F^*$ exists for a.e.~$\lambda.$ Numbers~$\lambda$ for which such limit exists we call $H_0$-regular. A number~$\lambda$ we call semi-regular,
if the limit $F (H_0 + F^*JF - \lambda - i0)^{-1}F^*$ exists for at least one bounded self-adjoint operator $J$ on $\mathcal K;$ otherwise we call~$\lambda$ essentially singular. 

In this paper I discuss essentially singular points. In particular, I give different conditions which ensure that a real number~$\lambda$ is essentially singular,
and discuss their relation to eigenvalues of infinite multiplicity which are known examples of essentially singular points. 
\end{abstract}
\maketitle

\tableofcontents

\bigskip 

\section{Introduction}

The limiting absorption principle (LAP) is an important tool in the spectral theory of self-adjoint operators, see e.g. \cite{AMG, BW, Kur,YaBook}.
It asserts that a properly regularised resolvent $(H-z)^{-1}$ of a self-adjoint operator $H$ converges in an appropriately chosen topology as the spectral parameter $z = \lambda + iy$
approaches the essential spectrum, $\sigma_{ess},$ of $H.$ The regularisation of the resolvent is achieved by a choice of a  \emph{rigging} in the Hilbert space. 
A rigging in an abstract Hilbert space $\hilb$ can come in the form of a triple of spaces
$X_- \subset \hilb \subset X_+$ with continuous embeddings, where $X_\pm$ are in general locally convex spaces but usually Banach spaces, or in the form of a fixed operator acting on $\hilb.$
Many natural Hilbert spaces are endowed with a natural choice of a rigging, which may however depend on some parameters. For instance, in $L^2(\mbR^d, dx)$ a natural rigging is given by the operator of multiplication by 
a function $(1 + |x|^2)^{-\delta/4}, \delta > 0,$ which corresponds to a triple of spaces $$L^2(\mbR^d,  (1 + |x|^2)^{\delta/2}\,dx) \  \subset \ L^2(\mbR^d,\,dx)  \ \subset \ L^2(\mbR^d,\,(1 + |x|^2)^{- \delta/2}dx ),$$
see e.g.~\cite{Agm}, but two particular values of $\delta$ are special: $\delta = 1$ and $\delta = d.$

There are versions of LAP which differ in other aspects, such as a choice of the  topology of convergence. Another point of difference is whether the convergence is supposed to be everywhere in the essential spectrum or a.e.,
and in the former case whether the convergence is uniform on compact subsets of $\sigma_{ess}$ and moreover whether the regularised resolvent admits analytic continuation through the essential spectrum. 
Examples of operators from this scenario include Schr\"odinger operators with short range potentials and periodic Schr\"odinger operators, see e.g. \cite{Agm,BeShu}.
But there are types of self-adjoint operators for which LAP can hold only almost everywhere.  Random Schr\"odinger operators are an important class of operators of this type.

We will say that LAP holds at a particular point $\lambda \in \mbR,$ if the regularised resolvent $(H-\lambda-iy)^{-1}$ converges as $y \to 0^+;$ in this case we will also say that~$\lambda$ is a regular point for $H.$
Even if LAP fails for $H$ at a point~$\lambda,$ it is possible that for an admissible, --- in a sense depending on the rigging, perturbation of $H$ LAP holds at~$\lambda.$ Such points we will call \emph{semi-regular}. 
Semi-regular points were introduced in \cite{AzSFIES} where they were called \emph{essentially regular}.
Semi-regularity of a point~$\lambda$ 
is a property of not a single self-adjoint operator, but rather of a space of operators, which consists of all admissible perturbations of $H.$ A point~$\lambda$ which is not semi-regular, we will call \emph{essentially singular}.

Given a norm-continuos path of self-adjoint operators with a common essential spectrum and a point~$\lambda$ outside the essential spectrum, there is a well-defined notion of spectral flow for the path which counts the net number of crossings  of the path eigenvalues through the point~$\lambda,$ --- this notion has its origins in geometry \cite{APSIII} and operator theory \cite{Kr53}. As it turns out, for a semi-regular point~$\lambda$ inside the essential spectrum the notion of spectral flow admits a natural generalisation,
see \cite{AzSFIES,AzSFnRI,AzDaMN}.   
For this reason semi-regular points were studied in depth in \cite{AzSFIES} and other papers of that series, while essentially singular points were largely ignored. But essentially singular points are very interesting, --- in a way they are genuine singularities of a 
self-adjoint operator acting on a rigged Hilbert space, since essential singularity of a point~$\lambda$ is permanent, while singularity at a semi-regular point is transitory. 
In addition to this, their possible existence proves to be a hindrance for studying self-adjoint operators, and for this reason alone they cannot be ignored. The aim of this paper is to discuss  essentially singular points, and to try to identify different reasons which  cause a real number to be essentially singular. 

In this paper I define four types of real numbers, associated  with a real affine space of self-adjoint  operators on a rigged Hilbert space. 
They are called $(I)$-singular, $(L)$-singular, $(N)$-singular and $(F)$-singular. It is proved that real numbers of each of these types are essentially singular. 
It is also shown that any $(I)$-singular number is $(F)$-singular. 

Definitions of these types of essentially singular numbers are as follows.
A real number is $(I)$-singular if it is an eigenvalue of infinite multiplicity of at least one self-adjoint operator from the space. A real number~$\lambda$ is $(L)$-singular, if infinitely many eigenvalues of the regularised resolvent, $T_z(H),$ 
of at least one operator $H$ from the space go to infinity as the spectral parameter $z$ approaches $\lambda + i0$ non-tangentially. $(N)$-singular numbers are defined similarly using $s$-numbers of a compact operator instead of its eigenvalues. 
Finally, $(F)$-singular numbers~$\lambda$ are defined using a regularised generalisation of the notion of eigenspace of a number~$\lambda,$ -- if the eigenspace is infinite-dimensional for at least one operator from the affine space 
then~$\lambda$ is $(F)$-singular. 

How these types of essentially singular numbers relate to each other is an open problem, but I believe that $(L)$-, $(N)$- and $(F)$-singular numbers have a good chance to be the same.

\section{Preliminaries}
A \emph{rigged Hilbert space} is a pair $(\hilb,F)$ which consists of a (complex separable) Hilbert space $\hilb,$ 
endowed with a rigging in the form of a zero-kernel and zero co-kernel closed operator $F$ from $\hilb$ to possibly another Hilbert space~$\clK,$
which we call an auxiliary Hilbert space, despite the fact that the majority of operators considered here will act on~$\clK.$ This is not the only way to introduce a rigging into a Hilbert space, but it is an easy to work with version  which suits our needs. 
We say that a self-adjoint operator, $H_0,$ on $\hilb$ acts on a rigged Hilbert space $(\hilb,F),$ if $F$ is $|H_0|^{1/2}$-compact, which means that $F (1+|H_0|)^{-1/2}$ is compact.
Or, in this situation we may also say that a self-adjoint operator $H_0$ and a rigging $F$ are compatible. 
In this case the sandwiched resolvent,
$$
   T_z(H_0) := F R_z(H_0) F^*
$$
is a compact operator on~$\clK.$ Further, let $\clA_0$ be the real vector space of self-adjoint operators of the form $F^* J F,$ where $J$ is a bounded self-adjoint operator on~$\clK,$
and let 
$$
    \clA := H_0 + \clA_0
$$
be an associated real affine space of self-adjoint operators. 
In this case for all operators $H$ from $\clA$ the sandwiched resolvent $T_z(H)$ is also compact. We will postulate that all operators from $\clA$ act on the rigged Hilbert space, in the sense above.
This is not a big restriction, see  \cite{DaThesis}.

By Weyl's theorem, all operators from $\clA$ share a common essential spectrum, which we denote $\sigma_{ess}(\clA),$ or simply $\sigma_{ess},$ since we will be working with a fixed affine space $\clA.$

We say that an operator $H$ from $\clA$ obeys the \emph{limiting absorption principle (LAP)}, if the norm limit 
\begin{equation} \label{F: T(lamb+i0)}
    T_{\lambda + i0} (H) := \lim_{y \to 0^+}     T_{\lambda + iy} (H)
\end{equation}
exists for almost every real number~$\lambda.$ If for a real number~$\lambda$ the limit $T_{\lambda + i0} (H)$ exists then we say that~$\lambda$ is $H$-regular, otherwise we call it $H$-singular. 
We may also say that in this case $H$ is~$\lambda$-regular, or otherwise that it is~$\lambda$-singular. 
It is a well-known and simple result that the complement of the set $\Lambda(H)$ is a core of singular spectrum, in the sense that the operator $H E_{\Lambda(H)} (H)$ is absolutely continuous. 
In particular, the set $\Lambda(H)$ does not contain eigenvalues of $H.$ For this reason, the terminology $H$-singular for a real number $\lambda \notin \Lambda(H)$ is justified. 

We note the following theorem from \cite{AzLAP}.

\begin{thm} \label{T1}
 If at least one operator from $\clA$ obeys LAP, then so does every other operator from $\clA.$
\end{thm}
This theorem shows that LAP is an attribute of the whole real affine space~$\clA$ of operators, rather than being property of a single self-adjoint operator. 
Let $\Lambda(H) = \Lambda(H,F)$ be the set of all $H$-regular numbers~$\lambda.$ Theorem \ref{T1} asserts that if $\Lambda(H)$ is a full set for one operator $H$ from $\clA$
then it is a full set for any operator from $\clA.$ Let 
$$
    \Lambda(\clA) = \bigcup_{H\in \clA} \Lambda(H). 
$$
Real numbers from the set $\Lambda(\clA)$ we call \emph{semi-regular}, or \emph{essentially regular}, or \emph{$\clA$-regular}. 
Real numbers from the complement to the set $\Lambda(\clA)$ we call \emph{essentially singular}.

The definition of a semi-regular number~$\lambda$ asserts that for at least one $H$ from~$\clA$ 
the limit  \eqref{F: T(lamb+i0)} exists. 
In this regard, the following theorem holds, see \cite{AzSFIESIV}.
\begin{thm}  \label{T2}
Let $H_0 \in \clA$ and $\lambda \in \mbR.$ If~$\lambda$ is semi-regular, then the limit 
$T_{\lambda+i0}(H_0 + rF^*F)$ exists for all $r \in \mbR$ except a discrete set. 
\end{thm}
This theorem shows that if~$\lambda$ is semi-regular, then every point of an affine line of the form $H_0 + \mbR F^*F$ in the affine space $\clA,$ except a discrete set, is~$\lambda$-regular. 
Further, for a semi-regular number~$\lambda$ the set, $\euR(\lambda),$ of~$\lambda$-singular numbers in $\clA$ is an analytic variety. This variety has co-dimension $1$ in the case of $\lambda \notin \sigma_{ess}(\clA),$ 
and co-dimension $>1$ in the case of $\lambda \in \sigma_{ess}(\clA),$ see \cite{AzSFnRI, AzSFIESII}.

As the following theorem shows, see \cite{AzSFIES},  essentially singular numbers exist.

\begin{thm}   \label{T: infinite mult eigenvalue}
If a real number~$\lambda$ is an eigenvalue of infinite multiplicity for at least one operator from $\clA$ then it is essentially singular. 
\end{thm}

Finally, a remark of terminological character: a real number is often called a point, but since the word ``point'' will also be used to denote an element of the affine space $\clA,$
elements of $\mbR$ will more often be called real numbers. 

\section{Essentially singular numbers}

Semi-regular real numbers~$\lambda$ were studied in depth in \cite{AzSFIES} and in subsequent papers of the series. 
What makes the semi-regular numbers~$\lambda$ interesting is that the notion of \emph{spectral flow} through~$\lambda$ admits a natural extension to such numbers, see \cite{AzSFIES}. 
At the same time, essentially singular numbers were largely ignored. 
But this is not because they are uninteresting, on the contrary --- they are in a sense genuine singularities of a LAP-obeying self-adjoint operator on a rigged Hilbert space.

Roughly speaking, there can be two reasons for a real number to be essentially singular. The first reason is that there can be too much singular spectrum near or at~$\lambda.$
The second reason is that the spectrum of a self-adjoint operator from $\clA$ is erratic near~$\lambda.$ Both scenarios depend on a choice of rigging in the Hilbert space:
the same real number can be essentially singular for one rigging and semi-regular for another. 
According to Theorem \ref{T: infinite mult eigenvalue},
one exception from this is, of course, an eigenvalue of infinite multiplicity.
In this regard, I state my first conjecture.

\begin{conj} Let $H$ be a self-adjoint operator on a Hilbert space $\hilb$ and~$\lambda$ be a real number. If~$\lambda$ is not an eigenvalue of infinite multiplicity for $H,$
then for some  compatible with $H$ choice of rigging $F$ in $\hilb,$ the number~$\lambda$  is semi-regular. 
\end{conj}

\begin{defn} A real number~$\lambda$ is $(I)$-singular, if it is an eigenvalue of infinite multiplicity for some $H \in \clA.$
\end{defn}
By Theorem \ref{T: infinite mult eigenvalue}, an $(I)$-singular number is essentially singular. 

\subsection{$(L)$-singular numbers}
Let~$\lambda$ be a real number and $H$ be a point of the affine space $\clA.$
We define a number $L(\lambda; H) \in [0, +\infty]$ as the supremum of numbers $R\geq 0$ such that 
for any $N \in \mb N$ and any $\eps>0$ there exists $y \in (0,\eps)$ such that at least $N$ eigenvalues 
of 
$
  T_{\lambda+iy}(H)
$
have absolute value $\geq R.$

The following theorem is probably intuitively obvious, but nevertheless I provide its proof. 
\begin{thm}
For any semi-regular point~$\lambda$ and any $H \in \clA$ we have $L(\lambda; H) = 0.$
\end{thm}
\begin{proof}
We write $H_0$ for $H.$
For $H_0$-regular~$\lambda$ the limit $T_{\lambda+i0}(H_0)$ exists and is compact.  Hence, in this case 
 $L(\lambda; H_0) = 0.$
 
Let  now~$\lambda$ be $H_0$-singular and semi-regular.
Recall that eigenvalues of $T_{\lambda+iy}(H_s)$ are $(s-r^j_{\lambda+iy})^{-1},$ where $r^j_z$ are the resonance points 
of the pair $(H_0,F^*F)$ and $$H_s = H_0+sF^*F.$$
The eigenvalues of $T_{\lambda+iy}(H_0)$ are thus $(-r^j_{\lambda+iy})^{-1},$ and we need to show that as $y\to 0^+$ only finitely many of those 
can get out of $R\mbD.$

Let $R>0.$
By Theorem \ref{T2} the limit $T_{\lambda+i0}(H_s)$ exists for all $s$ except a discrete set. 
We can assume that for  $s = 1$ the limit $T_{\lambda+i0}(H_s)$ exists. 
The eigenvalues of  $T_{\lambda+i0}(H_1)$ are
$(1 - r^j_{\lambda+i0})^{-1}.$ Since  $T_{\lambda+iy}(H_1)$ is compact and norm continuously depends on $y \in [0,1],$
 we have 
$$
    n :=  \sup _{y \in [0,1]}  \#  \left\{ j \colon   |1-r^j_{\lambda+iy}|^{-1}   >  \frac{R}{1+R} \right\} < \infty.
$$
In other words, for any $y\in [0,1]$ the operator $T_{\lambda+iy}(H_0)$ cannot have more than $n$ eigenvalues outside the disk of radius $R/(1+R).$
It is easy to see that the inequality $|r^j_{\lambda+iy}|^{-1} > R$ implies the inequality 

$$
     |1 - r^j_{\lambda+iy}|^{-1}   >  \frac{R}{1+R}.
 $$
Thus, for any $y \in [0,1]$ the number of eigenvalues of $T_{\lambda+iy}(H_0)$ outside the open disk of radius $R$ is no more than $n.$
 Hence, 
 $L(\lambda; H_0) \leq R$ and since $R>0$ is arbitrary, we have $L(\lambda; H_0) = 0.$
\end{proof}

\begin{defn} A real number~$\lambda$ is $(L)$-singular, if for some $H \in \clA$ we have $L(\lambda, H) > 0.$ 
We also say that~$\lambda$ is a singular number of type $(L).$
\end{defn}
Thus, the following theorem holds. 
\begin{thm}
    Any $(L)$-singular number is essentially singular.    \quad $\Box$
\end{thm}

\begin{conj} If~$L(\lambda,H) > 0$ for one $H \in \clA$ then $L(\lambda,H') > 0$ for any other $H' \in \clA.$
\end{conj}

\begin{conj} $(I)$-singular number is $(L)$-singular. 
\end{conj}

Note that $(L)$-singular numbers exist, --- a trivial example is provided by a scalar operator $H = \lambda \Id$ and a compact rigging $F.$

\subsection{$(N)$-singular numbers}
Recall that $s$-numbers, $s_1(T), s_2(T), \ldots,$ of a compact operator $T$ on a Hilbert space are eigenvalues of  $|T|$ written in the order 
of decreasing magnitudes. The following approximating property holds, see \cite{GoKr}: for any $n=0,1,2,\ldots$
\begin{equation} \label{F: approx property of s-numbers}
    s_{n+1}(T) = \min_{K \in \mathfrak K_n} \| T - K \|,
\end{equation}
where $\mathfrak K_n$ is the set of all operators of rank $\leq n.$
For two compact operators $A$ and~$B$ the following inequality is a special case of Fan's inequality, see \cite{GoKr},
\begin{equation} \label{F: special case of Fans inequality}
    s_{n+1}(A+B) \leq s_n(A) + \| B \|.
\end{equation}

Let~$\lambda$ be a real number and $H$ be a point of the affine space $\clA.$
We define $N(\lambda; H) \in [0, +\infty]$ as the supremum of numbers $R\geq 0$ such that 
for any $n \in \mb N$ and any $\eps>0$ there exists $y \in (0,\eps)$ such that at least $n$ $s$-numbers 
of 
$
  T_{\lambda+iy}(H)
$
have absolute value $\geq R.$

\begin{thm} \label{T: for semi-regular N(lambda) = 0} 
For any semi-regular point~$\lambda$ and any $H \in \clA$ we have $N(\lambda; H) = 0.$
\end{thm}
\begin{proof} Let $H_0 = H$ and $H_s = H_0 + sF^*F.$ Since~$\lambda$ is semi-regular, by Theorem \ref{T2} the limit
$T_{\lambda + i0}(H_s)$ exists for all $s$ except a discrete set. Let $r_z^j$ be those coupling resonances of the pair $H_0,  F^*F,$
which converge to zero as $z = \lambda + iy$ approaches $\lambda + i0.$ Since $\lambda $ is semi-regular, there are only finitely many such resonance points.
Recall that $r_z^j$ are poles of the meromorphic function $s \mapsto T_z(H_s),$ see \cite{AzSFIES}. 
We have 
$$
    T_z(H_s) = \tilde T_z(H_s) + \Pi(z, s),
$$
where  $\tilde T_z(H_s)$ is holomorphic at $r_z^j$ and $\Pi(z, s)$ is the pole part corresponding to the coupling resonances $r_z^j.$
The pole part $\Pi(z,s)$ does not (necessarily) converge as $z \to \lambda  + i0,$ but its rank is $\leq N,$ --- the number of the coupling resonances $r_z^j$ including multiplicities,
while $\tilde T_z(H_s)$ has the norm limit $\tilde T_{\lambda + i0}(H_s)$ when $s = 0,$ see~\cite{AzSFIES}. 
Since $\tilde T_{\lambda + iy}(H_0),$ $y \in [0,1],$ are compact, 
from this, using the approximating property \eqref{F: approx property of s-numbers} of $s$-numbers and \eqref{F: special case of Fans inequality}, one can easily infer that~$N(\lambda,H_0) = 0.$ 
\end{proof}

\begin{defn} A real number~$\lambda$ is $(N)$-singular, if for some $H \in \clA$ we have $N(\lambda, H) > 0.$ 
We also say that~$\lambda$ is a singular number of type $(N).$
\end{defn}

Similar to $(L)$-singular numbers, from Theorem \ref{T: for semi-regular N(lambda) = 0} we immediately have 
\begin{thm}
    Any $(N)$-singular number is essentially singular.    \quad $\Box$
\end{thm}

$(N)$-singular numbers exist, --- a trivial example is provided by a scalar operator $H = \lambda \Id$ and a compact rigging $F.$
A little more sophisticated examples of the same nature are provided by the following theorem.

\begin{thm}  If a real number $\lambda$ is an isolated point of spectrum of infinite multiplicity for some $H \in \clA,$ then it is $(N)$-singular. 
\end{thm}
\begin{proof}
We shall prove that if~$\lambda$ is an eigenvalue of infinite multiplicity for some $H \in \clA$ and an isolated point of spectrum of $H,$ 
then $N(\lambda, H) = \infty.$
Let $a > 0$ be such that $(\lambda - a, \lambda + a)$ contains only one point of spectrum of $H,$ namely the number~$\lambda.$
Let $P$ be the orthogonal projection onto the eigenspace $\clV_\lambda$ corresponding to the eigenvalue $\lambda.$
It is easy to see that the operator $FPF^*$ is compact and has infinite rank. 
Let 
$$
    G  := \int _{\mbR \setminus (\lambda - a, \lambda + a)} x\, dE_H(x),
$$
so that $H = \lambda P + G.$ With this notation, \eqref{F: special case of Fans inequality} implies
\begin{equation*} 
   \begin{split}
        s_n (T_{\lambda + iy }(H))   & =          s_n (T_{\lambda + iy }(G) + i y^{-1} F P F^*)  \\
        & \geq y^{-1}  s_{n+1} (FPF^*) - \sup_ {y \in [0,1]}  \| T_{\lambda + iy }(G) \|.
   \end{split}
\end{equation*} 
Clearly,  $\sup_ {y \in [0,1]}  \| T_{\lambda + iy }(G) \|$ is finite. Also, since $FPF^*$ is an infinite rank compact operator, infinitely many of its $s$-numbers go to $\infty$ as $ y \to 0^+.$
By the estimate above, the same can be said about $s$-numbers of  $T_{\lambda + iy }(H).$ In other words, $N(\lambda, H) = \infty.$
\end{proof}

I believe that this argument, property modified, should also work to prove that an $(I)$-singular number is $(N)$-singular. This is work in progress. 
\medskip 

\noindent
{\bf Remark.} Definition of $(N)$-singular numbers admits a modification: instead of $T_z(H) = F R_z(H) F^*$ one can use $ F^* F R_z(H).$ Using the other product, $R_z(H) F^*F$
does not give anything new since $s_n(A) = s_n(A^*).$
For $(L)$-singular numbers none of these modifications give anything new since $\spec(AB) \cup \{0 \} = \spec(BA) \cup \{0 \}.$ 

\subsection{$(F)$-singular numbers}
Let  $F$ and $\clA$ be as above.
\begin{defn}  
A real number~$\lambda$ is $(F)$-singular, if for some $H \in \clA$ 
$$
    \dim \bigcap _{O \ni \lambda}  \overline{\ran{F E_O(H)} } = \infty,
$$
where the intersection is over all open sets containing $\lambda.$
We also say that~$\lambda$ is a singular number of type $(F).$
\end{defn}

Recall that given a semi-regular point~$\lambda$ and $H \in \clA,$ the vector space  $\Upsilon^1_\lambda(H)$ is the finite dimensional space of solutions to the equation
$$
    ( 1 - T_{\lambda + i0 } (H + F^*JF) J ) u = 0,
$$
where $J$ is any bounded self-adjoint operator on $\clK$ for which $H+F^*JF$ is~$\lambda$-regular. Note that such $J$ exists by definition of semi-regularity, and the vector space 
$\Upsilon^1_\lambda(H)$ is independent of the choice of $J,$ see \cite{AzSFIES}.  Also note that the equation above is nothing else but  the homogeneous Lippmann-Schwinger equation,
see e.g. \cite{Agm,BeShu,TayST} for more information about it.

\begin{thm}   \label{T: F E(O) subset Upsilon}
Let~$\lambda$ be a semi-regular point. Then for any $H \in \clA$ we have 
\begin{equation} \label{F: bigcap F E subset Upsilon}
     \bigcap _{O \ni \lambda}  \overline{\ran{F E_O(H)} } \subset \Upsilon^1_\lambda(H).
\end{equation}
\end{thm}
\begin{proof}
In the proof  I  write $H_0$ for $H.$ Since $\lambda$ is semi--regular, there exists a bounded self-adjoint operator $J$ on $\clK$ such that $H_1 := H_0 + V$ is~$\lambda$-regular, where $V =  F^*JF.$
Recall that this means by definition that $T_{\lambda+iy}(H_1)$ converges in norm as $y \to 0^+,$ with the limit denoted $T_{\lambda+i0}(H_1).$ 
Let $u$ be an element of 
$$
      \bigcap _{O \ni \lambda}  \overline{\ran{F E_O(H_0)} }. 
$$
We show that $u$ is a solution to the homogeneous Lippman-Schwinger equation 
$$
    ( 1 - T_{\lambda + i0 }(H_1) J) u = 0.
$$
by showing  that $\|u - T_{\lambda+i0}(H_1)Ju\| < \varepsilon$ for any $\varepsilon>0.$
For any $y>0$ and $f \in \hilb,$ we have
\begin{align*}\label{F: triangles}
 \|u - T_{\lambda+i0}(H_1)Ju\|  
 & \leq  \|u - Ff
 \\ & \qquad + Ff - T_{\lambda+iy}(H_1)JFf
\\&\qquad +  ( T_{\lambda+iy}(H_1) - T_{\lambda+i0}(H_1) ) J Ff
\\&\qquad + T_{\lambda+i0}(H_1)J ( Ff - u ) \|. 
 \\ 
 & \leq  \|u - Ff\|
 \\ & \qquad + \|Ff - T_{\lambda+iy}(H_1)JFf\| 
\\&\qquad + \|T_{\lambda+iy}(H_1) - T_{\lambda+i0}(H_1)\|\| J  F f \| 
\\&\qquad + \|T_{\lambda+i0}(H_1)J\| \|Ff - u\|. 
\end{align*}
Since $u \in \bigcap_{\lambda \in O} \overline{F E_O(H_0)\hilb},$ for any neighbourhood $O$ of~$\lambda$ there exists a vector~$f$ from $E_O(H_0)\hilb$ 
such that $$(1 + \|T_{\lambda+i0}(H_1)J\|) \|Ff - u\| < \eps/3.$$ 
This takes care of the first and the fourth summands in the estimate above. 

Also, since $H_1$ is $\lambda$-regular, 
for all small enough $y$  the third summand is less than $\eps/3.$ 

Now we estimate the second summand which would complete the proof.
For $y>0,$ we choose $O$ to be the $y$-neighbourhood of $\lambda.$
We have
\begin{equation} \label{F: |FR(y)|<}
\|FR_{\lambda+iy}(H_1)\|  
  = \| \left|R_{\lambda-iy}(H_1)F^*\right| \|
  = y^{-1/2}  \Big\| \sqrt{\Im T_{\lambda+iy}(H_1)} \Big\|.
\end{equation}
The second term we rewrite as follows: 
\begin{equation} \label{F: second term}
  \begin{split}
Ff - T_{\lambda+iy}(H_1)JFf 
  &= FR_{\lambda+iy}(H_1)\left[(H_1 - \lambda - iy)f - Vf \right] 
\\&= FR_{\lambda+iy}(H_1)\left[ ( H_0 - \lambda) f - i y f \right].
  \end{split}
\end{equation}
By (\ref{F: |FR(y)|<}) and $\lambda \in \Lambda(H_1,F),$
the term $iyFR_{\lambda+iy}(H_1)f$ vanishes as $y \to 0^+.$ 
Since  $f \in E_{ ( \lambda - y, \lambda + y) }(H_0)\hilb,$
we have $\| (H_0 - \lambda) f \| \leq y \| f \|.$ Using this, 
the norm of the other summand can be estimated as follows:  
\begin{equation*} 
  \begin{split}
       \norm{FR_{\lambda+iy}(H_1) ( H_0 - \lambda ) f }    &  \leq   \norm{FR_{\lambda+iy}(H_1)} \norm{  ( H_0 - \lambda ) f }  \\ 
        & \leq  \norm{FR_{\lambda+iy}(H_1)} \cdot y \| f \|.
  \end{split}
\end{equation*}
Hence, by \eqref{F: |FR(y)|<} the second summand also vanishes. 
\end{proof}

Theorem \ref{T: F E(O) subset Upsilon} allows to prove the following 
\begin{thm} Any $(F)$-singular real number is essentially singular. 
\end{thm}
\begin{proof} Assume the contrary, --- an $(F)$-singular number~$\lambda$ is semi-regular.
Then for some $H \in \clA$ the vector space 
$
      \bigcap _{O \ni \lambda}  \overline{\ran{F E_O(H)} }
$
has infinite dimension. Since~$\lambda$ is semi-regular, the vector subspace $\Upsilon^1_\lambda(H)$ is well-defined and has a finite dimension.
Hence, Theorem \ref{T: F E(O) subset Upsilon} gives us a contradiction. 
\end{proof}

\begin{thm} Any $(I)$-singular real number is $(F)$-singular. 
\end{thm}
\begin{proof} Let a real number~$\lambda$ be $(I)$-singular, that is, 
for some $H \in \clA$  the number~$\lambda$ is an eigenvalue of infinite multiplicity. Let  $\clV_\lambda$  be the corresponding eigenspace.
We have  
$$
     F \clV_\lambda \subset    \bigcap _{O \ni \lambda}  \overline{\ran{F E_O(H)} },
$$
and since $F$ has zero kernel, the left hand side has infinite dimension. Hence, so does the right hand side. 
\end{proof}

\begin{conj}  \label{C9}  For a semi-regular number~$\lambda$ 
and any $H \in \clA,$
the inclusion in \eqref{F: bigcap F E subset Upsilon} can be replaced by equality:
$$
     \bigcap _{O \ni \lambda}  \overline{\ran{F E_O(H)} }  =   \Upsilon^1_\lambda(H).
$$
\end{conj}

\noindent {\bf Remark.}  I know how to prove Conjecture \ref{C9} if 
$$
    \int_{\mbR} \| T_{ x + iy }  (H)  \| (1+|x|^2) \, dx < \infty
$$
uniformly in $y \in [0,1].$ In particular, this assumes that LAP holds for~$H.$

\section{More open problems}

In this section I state some open problems. 

\subsection{On relation of different types of essentially singular numbers}

\begin{enumerate}
   \item Find an example of an essentially singular number, which is not (a) $(I)$-singular, 
   (b) $(N)$-singular, (c) $(L)$-singular, (d) $(F)$-singular, or prove that such numbers do not exist. 
   \item Find all possible relations between these four types of essential singular numbers. For instance, (a) is it true that singular numbers of types $(N)$ and $(L)$  are identical?
   (b) is it true that every $(I)$-singular number is $(N)$-singular?
   (c) is it true that every $(I)$-singular number is $(L)$-singular?
\end{enumerate}

\subsection{$(R)$-singular numbers}
In addition to previous types of essentially singular numbers, I will define one more type. 
\begin{defn}  Let $\clA$ and $F$ be as above, and
let~$\lambda$ be a real number. We say that~$\lambda$ is a singular point of type $(R),$ or an $(R)$-singular number, if 
for at least one $H \in \clA,$ 
as $\lambda + iy$ approaches to $\lambda+i0$ at least one coupling resonance function
$r_j(z)$ of the pair $H, F^*F$ does not converge in the extended complex plane $\bar \mbC.$ 
\end{defn}
The definition of $(R)$-singular number probably requires a bit more care since the coupling resonance functions can have branching and absorbing points, see \cite{AzSFIESV}, 
but this is not a difficult problem. 
It is obvious that an $(R)$-singular number is essentially singular. 
I believe that $(R)$-singular numbers should be quite abundant,
but at the moment I don't have examples of $(R)$-singular numbers.

I believe the following conjecture is closely related to $(R)$-singular numbers.
\begin{conj} Let $\clA$ and $F$ be as above. 
Let $\lambda \in \mbR.$
If for some $H \in \clA$ and some $\psi \in \hilb$ the limit
$$
    \lim _ {y \to 0^+}  \Scal { \psi , T_{\lambda + iy} (H) \psi}
$$
does not exist in the extended complex plane, then the real number $\lambda$ is essentially singular. 
\end{conj}
If this conjecture is true, it would provide plenty of examples of essentially singular points. 
A knowledge of asymptotic behaviour of coupling resonances $r_j(\lambda+iy)$ and of their associated idempotents $P_{\lambda+iy}(r_j(\lambda+iy)),$
see \cite{AzSFIES},
 seems to be important in tackling this conjecture. 

\bigskip\bigskip 
{\it Acknowledgements.} I thank my wife for financially supporting me during the work on this paper.

\end{document}